\documentclass[a4paper,10pt,oneside]{amsart}

\usepackage[english]{babel}
\usepackage{hyperref}
\hypersetup{
	colorlinks,
	linkcolor={red!50!black},
	citecolor={blue!50!black},
	urlcolor={blue!80!black},
	breaklinks=true
}
\usepackage{breakurl}
\usepackage{tikz}
\usetikzlibrary{arrows,positioning,calc,patterns}
\usepackage{tikz-cd}

\usepackage{url}
\theoremstyle{plain}
\newtheorem{lemma}{Lemma}[section]

\newtheorem{theorem}[lemma]{Theorem}
\newtheorem*{question}{Question}
\newtheorem*{theorem-nonum}{Theorem}

\theoremstyle{remark}

\theoremstyle{definition}
\newtheorem*{convention}{Convention}

\newcommand{\p}{\mathbb{P}}
\newcommand{\C}{\mathbb{C}}
\newcommand{\R}{\mathbb{R}}

\title{On the number of perfect triangles with a fixed angle}

\author[M. Makhul]{Mehdi Makhul}
\address[Mehdi Makhul]{Johann Radon Institute for Computational and Applied Mathematics (RICAM), Austrian Academy of Sciences, Linz}
\email{mmakhul@risc.jku.at}
\subjclass[2010]{52C10, 14G05} 
\keywords{Perfect triangle, Rational distance set}

\begin{document}

\begin{abstract}
Richard Guy asked the following question: can we find a triangle with rational sides, medians and area? Such a triangle is called a \emph{perfect triangle} and no example has been found to date.  It is widely believed that such a triangle does not exist. Here we use the setup of Solymosi and de Zeeuw about rational distance sets contained in an algebraic curve, to show that for any angle $0<\theta < \pi$, the number of perfect triangles with an angle $\theta$ is finite. A \emph{rational median set} $S$ is a set of points in the plane such that for every three non collinear points $p_1,p_2,p_3$ in $S$ all medians of the triangle with vertices at $p_i$'s have rational length. The second result of this paper is that no irreducible algebraic curve defined over $\R$ contains an infinite rational median set. 
\end{abstract}
\maketitle 

\section{introduction}

A median of a triangle is a line segment joining a vertex to the midpoint of the opposite side. Finding a triangle with rational sides, medians and area was asked as an open problem by Richard Guy in \cite[D$21$]{Guy2004}. Such a triangle is called a  \emph{perfect triangle}. Various research has been done towards this question, but to date the problem remains unsolved. If we do not require the area to be rational, there are infinitely many solutions. Euler gave a parametrization of such 'rational triangles', in which all three medians were rational, see \cite{Buchholz2002}, however there are examples of triangles with three integer sides and three integer medians that are not given by the Euler parametrization. Buchholz in \cite{Buchholz2002} showed that every rational triangle with rational medians corresponds to a point on a one parameter elliptic curve. In the same vein Buchholz and Rathbun \cite{Buchholz1997} have shown the existence of infinitely many \emph{Heron triangles} with two rational medians, where a Heron triangle is a triangle that has side lengths and area that are all rationals.

A related, but slightly different problem is the Erd\"{o}s-Ulam problem. We say that a subset $S \subset \R^2$ is a \emph{rational distance set} if the distance between any two points in $S$ is a rational number.

In $1945$ Ulam posed the following question, based on a result of Anning-Erd\"{o}s \cite{Anning1945}. See \cite[Problem D$20$]{Guy2004}.

\begin{question}[Erd\"{o}s-Ulam]
Is there a rational distance set~$S$ in the plane~$\R^2$ that is dense for the Euclidean topology?  
\end{question}
Solymosi and de Zeeuw \cite{Solymosi2010} used Faltings' Theorem to show that a rational distance set contained in a real algebraic curve contains finitely many points, unless the curve has a component which is either a line or a circle. Furthermore, if a line (resp. circle) contains infinitely many points of a rational distance set, then it contains all but at most $4$ (resp. $3$) points of the set. 

Although this problem is still open, there are several conditional proofs that show that the answer to the Erd\"{o}s-Ulam question is no. Shaffaf \cite{Shaffaf2018} and Tao~\cite{Tao2014} independently used the weak Lang conjecture to give a negative answer to this question. Pasten \cite{Pasten2017} also proved that the $abc$ conjecture implies a negative solution to the Erd\"{o}s-Ulam problem. 

In the same circle of ideas, the weak Lang conjecture was used \cite{Makhul2012} to show that if $S$ is a rational distance set of $\R^2$ which intersects any line in only finitely many points, then there is a uniform bound on the cardinality of the intersection of $S$ with any line. 
Recently, Ascher, Braune and Turchet \cite{Ascher2019} considered rational distance sets~$S \subset \R^2$ such that no line contains all but at most four points of $S$, and no circle contains all but at most three points of $S$. They showed by assuming the weak Lang conjecture that there exists a uniform bound on the cardinality of such sets $S$.

Along the same lines, a \emph{rational median set} $S$ is a set of non-collinear points in $\R^2$ such that for every three non-collinear points $p_1, p_2$ and $p_3$ in $S$ all medians of the triangle with vertices at $p_i$'s have rational length. In a similar spirit to the Erd\"{o}s-Ulam question one might expect that if $S$ is a rational median set in the real plane $\R^2$, then $S$ must be very restricted, even a finite set.

Following the setup of Solymosi and de Zeeuw \cite{Solymosi2010}, in this paper we consider two problems. First, fix an angle $\theta$ and consider the number of perfect triangles such that one of their angles is $\theta$. The following is the first result.

\begin{theorem}\label{thm:perfect-triangle}
	Given $0<\theta < \pi$, up to similarity, there are finitely many perfect triangles with an angle $\theta$.
\end{theorem}
Our second result asserts that if $S$ is a rational median set, then every real algebraic curve intersects $S$ in finitely many points.
\begin{theorem} \label{thm:intersection-curve}
	Every rational median set in the plane $\R^2$ has finitely many points in common with an irreducible real algebraic curve.
\end{theorem}
\section{Preliminaries on genera of curves}

Given an affine algebraic curve in $\R^2$, defined by a polynomial~$f\in \mathbb{K}[x,y]$, ($\mathbb{K}$ is a subfield of $\R$) one can consider its projective closure, which is a projective algebraic curve, by taking the zero set of the homogenisation of $f$. This curve in $\p^2_{\R}$ then extends to $\p^2_{\C}$, by taking the complex zero set of the homogenised polynomial. In particular, when we consider the genus of a curve, we are talking about complex projective algebraic curves. 

To a given irreducible projective curve $X$ over complex numbers $\C$ we associate two invariants. One is the \emph{geometric genus} $g(X)$, and the other one is the \emph{arithmetic genus} $p_a(X)$. For more details on these notions we refer the reader to~ \cite[\href{https://stacks.math.columbia.edu/tag/0BYE}{Tag 0BYE}]{StacksProjectAuthors2018}.

If $X$ is a smooth complex algebraic curve, then the \emph{geometric genus} of $X$ is 
\[
g(X)=\dim_{\mathbb{C}}H^0(X,\Omega_{X}),
\]
where $\Omega_{X}$ is the canonical bundle on $X$. Moreover, since a smooth complex projective curve is a compact Riemann surface, the geometric genus coincides with the topological genus of the surface.  For a singular curve~$X$ we define the geometric genus to be the geometric genus of a smooth curve birational to~$X$. 

For any complex algebraic curve $X$ the \emph{arithmetic genus} of $X$ is defined as
\[
p_a(X)=1-\dim_{\mathbb{C}} H^0(X,\mathcal{O}_X)+\dim_{\mathbb{C}} H^1(X,\mathcal{O}_X).
\]
It is known that if $X$ is a smooth curve, then $p_a(X)=g(X)$.The arithmetic genus of a curve contained in a smooth surface with the canonical divisor $\mathcal{K}$ is given by (see \cite[Proposition 1.5, page 361]{Hartshorne1977} when $X$ is smooth)  
\begin{equation}\label{eq:arithmetic-genus}
p_a(X)=\frac{X\cdot(X+\mathcal{K})}{2}+1,
\end{equation}
where $\_ \cdot \_$ denotes the intersection product of the surface. For instance, if the surface is the projective plane $\p^2$ and $X$ is a planar curve of degree~$d$ we have~$X=dL$ and $\mathcal{K}=-3L$ where $L$ is the class of a line. Hence
\[
p_a(X)=\frac{dL\cdot(d-3)L}{2}+1= \frac{(d-1)(d-3)}{2}.
\]

Equation \eqref{eq:arithmetic-genus} enables us to compute the arithmetic genus for a reducible curve contained in a smooth surface with canonical divisor $\mathcal{K}$. In particular, if $X$ is a reducible curve with components $D_1, \dots, D_m$, substituting $X=D_1+\dots+D_m$ in Equation~\eqref{eq:arithmetic-genus}, we obtain 
\begin{equation}\label{eq:genus-reducible}
p_a(X)=\sum_{k=1}^{m}p_a(D_k)+ \sum_{i\not=j}D_i\cdot D_j- (m-1).
\end{equation}
\begin{convention}
In this paper, by \emph{genus} of a curve we will mean the geometric genus, unless
otherwise specified.
\end{convention}
The main ingredients in our proof are the following theorem of Faltings \cite{Faltings1984} and the Riemann–Hurwitz formula \cite[Theorem $5.9$]{Silverman1986}.

\begin{theorem}[Faltings]
	\label{thm:faltings}
	Let $K$ be a number field. If $X$ is an algebraic curve over~$K$ of genus $g\ge 2$, then the set $X(K)$ of $K$-rational points is finite.
\end{theorem}

\begin{theorem}[Riemann-Hurwitz]
	\label{thm:riemann-hurwitz}
	Let $\phi \colon X_1 \rightarrow X_2$ be a non-constant separable map of curves. Then
	\[
	2g_1-2 \ge (\deg \phi)(2g_2-2)+ \sum_{p\in X_1} (e_p-1),
	\]
	where $g_i$ is the genus of $X_i$ and $e_p$ is the ramification index of $\phi$ at $p$.
\end{theorem}
We will make use of the following result from \cite{Hegedues2015}.

\begin{lemma}\label{lm:josef-zija}
Let $Y_1$ and $Y_2$ be two smooth curves of genus at most $1$. Let $Y\subset Y_1 \times Y_2$ be an irreducible curve such that the two projections restricted to $Y$ are either birational or $2:1$ maps to $Y_1$ resp $Y_2$. Then
\begin{itemize}
		\item If $Y_1$ and $Y_2$ are rational curves, then $Y$ is a curve in $\p^1 \times \p^1$ of bi-degree $2$, which has arithmetic genus $1$. The geometric genus is $1$ in the nonsingular case and $0$ if $Y$ has a double point.
		\item If $Y_1$ is elliptic and $Y_2$ is rational, then the arithmetic genus of $Y$ is $3$.
		\item If $Y_1$ and $Y_2$ are both elliptic, then the genus of $Y$ is $5$.
	\end{itemize}
\end{lemma}

\section{Proof of Theorem \ref{thm:perfect-triangle} }

The proof of this theorem consists of two parts. First we assume $\theta = \frac{\pi}{2}$, since in this case the proof is different to other angles.

\textbf{Case $1$ if $\theta=\frac{\pi}{2}$:}
Suppose that $\Delta$ is a right triangle with rational sides and medians.  We show that there are finitely many such triangles. Without loss of generality we may assume the hypotenuse of $\Delta$ has length $b$ and the sides adjacent to the right angle have length $a$ and $1$. Let~$m_1, m_2$, and~$m_3$ denote the median lengths of $\Delta$; see Figure~\ref{fig:right-angle}. By the formulas expressing medians in terms of edges,
we have
\[
4m_1^2=2a^2+2b^2-1,\quad
4m_2^2=2a^2+2-b^2,\quad
4m_3^2=2b^2+2-a^2,
\]
where, $m_1, m_2$ and $m_3$ are rational numbers. On the other hand, by the Pythagorean Theorem we have $b^2=a^2+1$, and if we substitute this into the above formula we obtain
\[
4m_1^2=4a^2+1, \quad
4m_2^2=a^2+1,\quad
4m_3^2=a^2+4.
\] 
Define the curve $C$ in the $xz$-plane:
\[
C\colon z^2=(4x^2+1)(x^2+1)(x^2+4).
\]
If $a$ and $b$ are the length of a right angle triangle with all medians $m_1, m_2$ and $m_3$ having rational length, then we obtain a rational point $(a,z)=(a,2^3m_1m_2m_3)$ on $C$. On the other hand the roots of the right hand side are distinct, thus $C$ is a hyperelliptic curve of degree $6$ in the~$xz$-plane~\cite[Section 6.5]{Shafarevich2013}. Thus $C$ has genus two and by Faltings' Theorem it has finitely many rational points. This completes the proof when $\theta=\frac{\pi}{2}$.

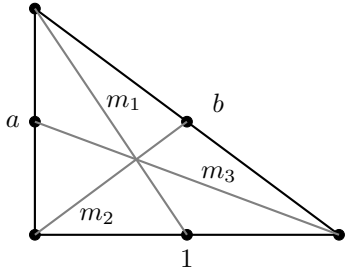
\begin{figure}
	\begin{tikzpicture}
	\filldraw [black] (0,0) node[anchor=west]{} circle (2pt);
	\filldraw [black] (4,0) node[anchor=west]{} circle (2pt);
	\filldraw [black] (0,3) node[anchor=west]{} circle (2pt);
	\filldraw [black] (2,0) node[anchor=west]{} circle (2pt);
	\filldraw [black] (0,1.5) node[anchor=west]{} circle (2pt);
	\filldraw [black] (2,1.5) node[anchor=west]{} circle (2pt);
	\draw[black, thick] (0,3) -- (4,0);
	\draw[black, thick] (0,0) -- (4,0);
	\draw[black, thick] (0,0) -- (0,3);
	\node[circle, anchor=north] (n1) at (2,0) {$1$};
	\node[circle, anchor=east] (n1) at (0,1.5) {$a$};
	\node[circle, anchor=west] (n1) at (2.1,1.75) {$b$};
	\node[circle, anchor=west] (n1) at (0.75,1.75) {$m_1$};
	\node[circle, anchor=west] (n1) at (2,0.8) {$m_3$};
	\node[circle, anchor=west] (n1) at (0.4,0.25) {$m_2$};
	\draw[gray, thick] (0,0) -- (2,1.5);
	\draw[gray, thick] (4,0) -- (0,1.5);
	\draw[gray, thick] (0,3) -- (2,0);
	\end{tikzpicture}
	\caption{A right angle with side lengths $1, a$, and $b$, and median lengths $m_1, m_2$, and $m_3$.}\label{fig:right-angle}
\end{figure}

\textbf{Case $2$ if $\theta\not=\frac{\pi}{2}$:}
Fix an angle $\theta\not=\frac{\pi}{2}$. Let $a, b$ denote the side lengths of a perfect triangle~$\Delta$ such that the angle between these two sides is $\theta$. Without loss of generality we may assume the side opposite $\theta$ has length $1$. Let $\lambda=\cos(\theta)$ and $\alpha=\sin(\theta)$. By the law of cosines we have 
\[
1=a^2+b^2-2\lambda ab.
\]
The rationality of the area of $\Delta$ and law of cosines guarantee that both $\lambda$ and $\alpha$ are rational numbers.
Let $X_0$ be the ellipse defined by\footnote{In general the conic $ax^2+bxy+cy^2+dx+ey+f=0$ is an ellipse if $b^2-4ac <0$. In our situation $b=2\lambda=2\cos(\theta)$, where $0<\theta<\pi$ and $\theta \not= \frac{\pi}{2}$ and $d=e=0$.}
\[
G(x,y)=1-x^2-y^2+2\lambda xy.
\]

Since $\Delta$ is a perfect triangle, all its medians $m_1, m_2, m_3$ and its area $s$ are rational. We have,
\begin{equation*}\label{eq:median+area}
\begin{aligned}
4m_1^2&=2a^2+2b^2-1,\\
4m_2^2&=2a^2+2-b^2,\\
4m_3^2&=2b^2+2-a^2,\\
s&=\frac{ab \alpha}{2}.
\end{aligned}
\end{equation*}
On the other hand $G(a,b)=0$, so for every perfect triangle as above we obtain a rational point $(a,b,m_1,m_2,m_3,s)$ on the curve $X_{\alpha}$ in $\R^6$, given by
\begin{equation*}\label{eq:perfect-curve}
\begin{aligned}
G(x,y)&=0,\\
4t_1^2-2x^2-2y^2+1&=0,\\
4t_2^2-2x^2-2+y^2&=0,\\
4t_3^2-2y^2-2+x^2&=0,\\
w-\frac{xy \alpha}{2}&=0.
\end{aligned}
\end{equation*}
We shall show that the genus of $X_{\alpha}$ is strictly bigger than $1$. To do that consider the curves 
\begin{equation*}\label{eq:space-curve}
\begin{aligned}
X_1&=\left\{(x,y,t_1): G(x,y)=0,\, 4t_1^2-2x^2-2y^2+1=0 \right\},\\
X_2&=\left\{(x,y,t_2): G(x,y)=0,\, 4t_2^2-2x^2-2+y^2=0 \right\},\\
X_3&=\left\{(x,y,t_3): G(x,y)=0,\, 4t_3^2-2y^2-2+x^2=0 \right\}.\\
\end{aligned}
\end{equation*}
Define the curve $X_{12}$ in $\R^4$, by
\begin{equation*} \label{eq:space-curve2}
\begin{aligned}
G(x,y)&=0,\\
4t_1^2-2x^2-2y^2+1&=0,\\
4t_2^2-2x^2-2+y^2&=0.
\end{aligned}
\end{equation*}
Similarly, we may define $X_{13}$ and $X_{23}$. By considering the Jacobian matrix of $X_1$ we can see $X_1$ is a smooth curve (even in the projective space $\p^3$), hence the geometric genus of $X_1$ is equal to the arithmetic genus of $X_1$. Now we show that the geometric genus of $X_1$ is $1$.

Consider the projection map~$\pi_1 \colon X_1 \rightarrow X_0$ given by
\[(x,y,t_1) \mapsto (x,y).\]
The preimage of each point $(x,y) \in X_0$ contains two points $(x,y, \pm t_1)$ in $X_1$ except when $t_1=0$. Hence $\pi_1$ is a map of degree $2$. By applying the Riemann-Hurwitz formula we can bound the genus of $X_1$ from below. In particular
\[
2g(X_1)-2 \ge \deg(\pi_1)(2g(X_0)-2)+ \sum_{p\in X_1}(e_p-1).
\] 
Since the genus of $X_0$ is zero (it is a conic), we have
\[
g(X_1) \ge -1+\frac{1}{2}\sum_{p\in X_1}(e_p-1).
\]
So to get $g(X_1) \ge 1$, we need to show that the projection $\pi_1$ has at least three ramification points. The potential ramification points correspond to the preimages of the intersection of $X_0=V(G(x,y))$ with the conic $2x^2+2y^2-1=0$, where by Bezout's Theorem there are $4$ such points, counting with multiplicities. 
\[
2x^2+2y^2-1=0,\quad x^2+y^2-1-2\lambda xy=0.
\]
By computing the discriminant we can see that this circle and ellipse intersect at $4$ distinct points. Therefore, we get $4$ ramification points. Thus Riemann-Hurwitz implies that the genus of~$X_1$ is at least~$1$. On the other hand, $X_1$ is a smooth space curve of degree $4$, so its genus is at most $1$. Hence $g(X_1) = 1$.  

\textbf{Claim: $X_1$ is irreducible}

The proof is by contradiction. Suppose that $X_1$ is a reducible curve, and $D_1,\dots, D_m$ are its irreducible components, then by the Equation \eqref{eq:genus-reducible} we know that the arithmetic genus~$p_a(X_1)$ is
\[
p_a(X_1)=\sum_{k=1}^{m}p_a(D_k)+ \sum_{i\not=j}D_i\cdot D_j- (m-1),
\]
where $D_i \cdot D_j$ is the intersection of the components $D_i$ and $D_j$. On the other hand, we have seen that $X_1$ is smooth, hence its geometric genus is equal to the arithmetic genus. Moreover, its irreducible components do not intersect. Hence~$p_a(X_1)=g(X_1)=1$, which implies that the number of irreducible components of $X_1$ is at most two. However, the degree of $X_1$ is $4$, thus if $X_1$ is reducible then it must be the union of an elliptic curve $E$ and a line $l$ that does not intersect $E$. Therefore, 
\[
p_a(X_1)=p_a(E)+p_a(l)-1,
\]
and this is a contradiction. Hence, $X_1$ is irreducible. A similar argument implies that $X_2$ is also irreducible.

\textbf{Claim: $X_{12}$ is an irreducible curve}

Consider two $2:1$ projection maps $\pi_1 \colon X_1 \rightarrow X_0$ and $\pi_2 \colon X_2 \rightarrow X_0$ defined by $\pi_1((x,y,t_1))=(x,y)$ and $\pi_2((x,y,t_2))=(x,y)$ respectively. The curve $X_{12}$ is also given as follows,
\[
X_{12}\colon=\left\{(p_1,p_2) \in X_1 \times X_2: \pi_1(p_1)=\pi_2(p_2) \right\}.
\]
$X_{12}$ is the fiber product of $X_1$ and $X_2$. By \cite[Theorem $3.3$ page $86$]{Hartshorne1977} since $X_1$ and $X_2$ are irreducible, $X_{12}$ is irreducible, unless the two projection maps $\pi_1$ and $\pi_2$ have some branching points in common. The branching points of $\pi_1$ are in the form $(x_i,y_i)$, where $x_i$ and $y_i$ satisfy
\begin{equation*}
\begin{aligned}
x_i^2&=\frac{\lambda+ \sqrt{\lambda^2-1}}{4\lambda}, \quad & y_i^2&=\frac{1}{4\lambda(\lambda+\sqrt{\lambda^2-1})},\, \text{for}\, i=1, 2,\\
x_i^2&=\frac{\lambda-\sqrt{\lambda^2-1}}{4\lambda}, \quad& y_i^2&=\frac{1}{4\lambda(\lambda-\sqrt{\lambda^2-1})},\, \text{for} \, i=3, 4.
\end{aligned}
\end{equation*}
None of them are a branching point of $\pi_{2}$. In particular, this implies that $X_{12}$ is a smooth curve\footnote{https://math.stackexchange.com/questions/1479139/fiber-products-of-curves}. 
Now by applying Lemma~\ref{lm:josef-zija} (see~\cite[Lemma $3$]{Hegedues2015}) $X_{12}$ is an irreducible curve with genus~$5$, unless the two projection maps~$\phi_1$ and $\phi_2$ from $X_{12}$ to~$X_1$ and $X_2$ respectively, defined by: 
\[
\phi_1(x,y,t_1,t_2)=(x,y,t_1) \quad \text{and} \quad \phi_2(x,y,,t_1,t_2)=(x,y,t_2),
\]
are birational. But this is not the case, in fact by considering the following commutative diagram

\[ \begin{tikzcd}
X_{12} \subset X_1 \times X_2 \arrow{r}{\phi_1} \arrow[swap]{d}{\phi_2} & X_1 \arrow{d}{\pi_1} \\%
X_2 \arrow{r}{\pi_2}& X_0
\end{tikzcd}
\]
we have $\pi_1=\pi_2 \circ \phi_2 \circ \phi_1^{-1}$, thus if $\phi_1$ and $\phi_2$ are birational, it implies that $\pi_1$ and $\pi_2$ must have the same branching points (since $\phi_2 \circ \phi_1^{-1}$ is an isomorphism) and this is a contradition. Hence the genus of $X_{\alpha}$ is at least five, and by Faltings' Theorem, $X_{\alpha}$ has finitely many rational points, and this completes the proof.\hfill$\square$

\section{Proof of Theorem \ref{thm:intersection-curve}}

In the following lemma we will see that for a rational median set, we are always able to apply a rotation, rational scaling or transformation (that preserves the rationality of distances) to see that the rational median set has a simple form. This is an analogue to \cite[Lemma $2.2$ $3$]{Makhul2012} for rational distance sets.
\begin{lemma}\label{lm:coordinates}
	Suppose that $S$ is a rational median set. Then there exists a square free integer $k$ such that if a similarity transformation $T$ transforms two points of $S$ in to $(0,0)$ and $(1,0)$ then any point in $T(S)$ is of the form 
	\[
	(r_1,r_2\sqrt{k}), \quad r_1,r_2 \in \mathbb{Q}.
	\]
\end{lemma}

\begin{proof}
Let $S^{\prime}=T(S)$. Let $(0,0)$, $(1,0)$ and $(x,y)$ be three non-collinear points in $S'$, then by the assumption the distance between $(x,y)$ and $(\frac{1}{2},0)$ is a rational number. Similarly, the distance between $(1,0)$ and $(\frac{x}{2},\frac{y}{2})$ is also a rational number. Specifically,
\begin{equation*}
\left(x-\frac{1}{2}\right)^2+y^2=r_1^2 , \quad \left(\frac{x}{2}-1\right)^2+\left(\frac{y}{2}\right)^2=r_2^2,
\end{equation*}   
where $r_1,r_2 \in \mathbb{Q}$. By eliminating $y$ from these two equations, we have 
\begin{equation*}
\left(x-\frac{1}{2}\right)^2-\left(x-2\right)^2=r_1^2-4r_2^2, \quad \text{hence} \quad x=\frac{r_1^2-4r_2^2}{6}+\frac{5}{4}.
\end{equation*} 
Therefore, $x$ is a rational number. A simple manipulation shows that $y=r\sqrt{k}$ where $r \in \mathbb{Q}$ and $k$ is a square free integer.
	
For the uniqueness of $k$, suppose that $p_1=(r_1,r_2\sqrt{k})$ and $p_2=(r_3,r_4\sqrt{k^{\prime}})$ are in~$S'$. By assumption, the distance between the origin and the middle point $\left(\frac{r_1+r_3}{2},\frac{r_2\sqrt{k}+r_4\sqrt{k^{\prime}}}{2}\right)$ is a rational number (consider the triangle with vertices $p_1, p_2$ and the origin). Hence the number~$2r_2r_4\sqrt{kk^{\prime}}$ should be rational, therefore $k=k^{\prime}$, since $k$ and $k^{\prime}$ are squarefree.
\end{proof}	
Notice that if $C$ is a curve of degree $d$ which contains more than $\frac{d(d+3)}{2}$ points from a rational median set $S$, then $C$ is defined over $\mathbb{Q}(\sqrt{k})$.

\textbf{Proof of Theorem \ref{thm:intersection-curve}}

Similar to the proof of Theorem \ref{thm:perfect-triangle}, we split the proof of this theorem into two parts. First we assume the real algebraic curve $C$ is a line, since in this case the proof is different to higher degree curves.

\textbf{Case $1$: $C$ is a line:}

Suppose that we have a rational median set $S$ with infinitely many points on a line. By definition, we have at least one point off that line. Without loss of generality we may assume the $x$-axis contains infinitely many points of $S$ and $(a,b)$ is a point of $S$ that is off the $x$-axis. Take three points $(c_1,0)$ $(c_2,0)$ and $(c_3,0)$ of $S$ on the $x$-axis. Then we have that for every point $(x,0)$ of $S$ on the $x$-axis (see Figure \ref{fig:almost-collinear}).
\[
\left(x-\frac{a+c_1}{2}\right)^2+\frac{b^2}{4}, \quad \left(x-\frac{a+c_2}{2}\right)^2+\frac{b^2}{4},\quad \text{and} \quad \left(x-\frac{a+c_3}{2}\right)^2+\frac{b^2}{4}
\]
are rational squares (to see this just consider the medians of a triangle with vertices at $(x,0)$, $(c_i,0)$ and $(a,b)$ for $i=1,2,3$). Thus we get a rational point $(x,y)$ on the curve 
\[
y^2=\prod_{i=1}^{i=3}\left[ \left(x-\frac{a+c_i}{2}\right)^2+\frac{b^2}{4}\right]
\]
This is a curve of genus two, since we may choose $(c_1,0)$, $(c_2,0)$ and $(c_3,0)$ such that all roots of the right-hand side are distinct. Therefore by Faltings' Theorem~\ref{thm:faltings}, the curve cannot contain infinitely many rational points, contradicting the fact that $S$ has infinitely many points on the $x$-axis.

\begin{figure}
	\begin{tikzpicture}[scale=0.8]
	\coordinate (A) at (-3,0);
	\coordinate (B) at (-2,0);
	\coordinate (C) at (-1,0);
	\coordinate (D) at (0,0);
	\coordinate (E) at (1,0);
	\coordinate (F) at (2,0);
	\coordinate (G) at (3,0);
	\coordinate (H) at (5,0);
	\coordinate (K) at (0.5,1.5);
	\draw (3,0) node[anchor=north]{$(c_i,0)$};
	\draw (0,2) node[anchor=west]{$(a,b)$};
	
	\path (A) -- (B) coordinate[pos=-1](dd) coordinate[pos=1.5](ff);
	\path (B) -- (C) coordinate[pos=-1](dd) coordinate[pos=1.5](ff);
	\path (C) -- (D) coordinate[pos=-1](dd) coordinate[pos=1.5](ff);
	\path (D) -- (E) coordinate[pos=-1](dd) coordinate[pos=1.5](ff);
	\path (E) -- (F) coordinate[pos=-1](dd) coordinate[pos=1.5](ff);
	\path (A) -- (G) coordinate[pos=-1](dd) coordinate[pos=1.5](ff);
	\draw[black] (dd) -- (A)node{$\bullet$}-- (B) node {$\bullet$}--(ff);
	\draw[black] (dd) -- (B)node{$\bullet$}-- (C) node {$\bullet$}--(ff);
	\draw[black] (dd) -- (C)node{$\bullet$}-- (D) node {$\bullet$}--(ff);
	\draw[black] (dd) -- (D)node{$\bullet$}-- (E) node {$\bullet$}--(ff);
	\draw[black] (dd) -- (E)node{$\bullet$}-- (F) node {$\bullet$}--(ff);
	\draw[black] (dd) -- (A)node{$\bullet$}-- (G) node {$\bullet$}--(ff);
	\draw[black] (dd) -- (A)node{$\bullet$}-- (H) node {$\bullet$}--(ff);
	
	\filldraw [black] (1,1) circle (2pt);
	\filldraw [black] (0.5,1) circle (2pt);
	\filldraw [black] (1.5,1) node[anchor=west]{$(\frac{a+c_i}{2},\frac{b}{2})$} circle (2pt);
	
	\coordinate (A1) at (0,2);
	\coordinate (B1) at (0,0);
	\path (A1) -- (B1) coordinate[pos=-1](dd) coordinate[pos=1.5](ff);
	\draw[black] (dd) -- (A1)node{$\bullet$}-- (B1) node {$\bullet$}--(ff);	
	
	\coordinate (A2) at (0,2);
	\coordinate (B2) at (1,0);
	\path (A2) -- (B2) coordinate[pos=-1](dd) coordinate[pos=1.5](ff);
	\draw[black] (dd) -- (A2)node{$\bullet$}-- (B2) node {$\bullet$}--(ff);

	\coordinate (A3) at (0,2);
	\coordinate (B3) at (2,0);
	\path (A3) -- (B3) coordinate[pos=-1](dd) coordinate[pos=1.5](ff);
	\draw[black] (dd) -- (A3)node{$\bullet$}-- (B3) node {$\bullet$}--(ff);

	\coordinate (A4) at (0,2);
	\coordinate (B4) at (3,0);
	\path (A4) -- (B4) coordinate[pos=-1](dd) coordinate[pos=1.5](ff);
	\draw[black] (dd) -- (A4)node{$\bullet$}-- (B4) node {$\bullet$}--(ff);

	\coordinate (A5) at (1,1);
	\coordinate (B5) at (3,0);
	\path (A5) -- (B5) coordinate[pos=-1](dd) coordinate[pos=1.5](ff);
	\draw[gray] (dd) -- (A5)node{$\bullet$}-- (B5) node {$\bullet$}--(ff);
	
	\coordinate (A6) at (1,1);
	\coordinate (B6) at (1,0);
	\path (A6) -- (B6) coordinate[pos=-1](dd) coordinate[pos=1.5](ff);
	\draw[gray] (dd) -- (A6)node{$\bullet$}-- (B6) node {$\bullet$}--(ff);
	
	\coordinate (A7) at (0.5,1);
	\coordinate (B7) at (3,0);
	\path (A7) -- (B7) coordinate[pos=-1](dd) coordinate[pos=1.5](ff);
	\draw[gray] (dd) -- (A7)node{$\bullet$}-- (B7) node {$\bullet$}--(ff);
	
	\coordinate (A8) at (1.5,1);
	\coordinate (B8) at (1,0);
	\path (A8) -- (B8) coordinate[pos=-1](dd) coordinate[pos=1.5](ff);
	\draw[gray] (dd) -- (A8)node{$\bullet$}-- (B8) node {$\bullet$}--(ff);
	\end{tikzpicture}
	\caption{The point $(\frac{a+c_i}{2},\frac{b}{2})$ is the middle point of $(a,b)$ and $(c_i,0)$ .}\label{fig:almost-collinear}
\end{figure}
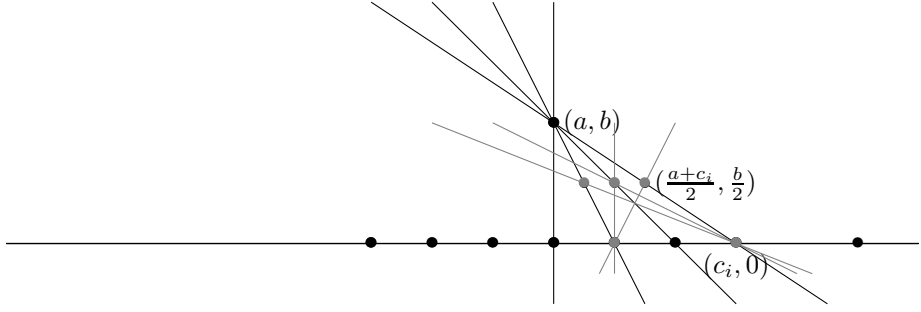

\textbf{Case $2$: $C$ is a curve of degree $d \ge 2$.}

Let~$C\colon=F(x,y)=0$ be an irreducible algebraic curve of degree~$d\ge2$. Suppose that there exists an infinite rational median set~$S$ contained in~$C$. We may assume $(0,0)$ and $(1,0)$ are on $S$. Hence by Lemma \ref{lm:coordinates} the elements of $S$ are of the form $(r_1,r_2\sqrt{k})$. If the genus of $C$ is at least~$2$, then by Faltings' Theorem \ref{thm:faltings}~$S$ is a finite set. 

From now on we assume $C$ is a curve of degree $d \ge 2$ and genus $0$ or $1$. Fix $p_1=(a_1,b_1)$ and $p_2=(a_2,b_2)$ in $S$. For an arbitrary point $(x,y) \in S$ that is not collinear with $p_1,p_2$ we have a triangle such that all its medians $m_1, m_2, m_3$ are rational, see Figure~\ref{fig:triangle}. We have,
\begin{equation*}\label{eq:medians}
\begin{aligned}
m_1^2= \Big(\frac{2x-a_1-a_2}{2}\Big)^2+\Big(\frac{2y-b_1-b_2}{2}\Big)^2,\\
m_2^2=\Big(\frac{x+a_2-2a_1}{2}\Big)^2+\Big(\frac{y+b_2-2b_1}{2}\Big)^2 ,\\
m^2_3=\Big(\frac{x+a_1-2a_2}{2}\Big)^2+\Big(\frac{y+b_2-2b_1}{2}\Big)^2,
\end{aligned}
\end{equation*}
where, $m_1, m_2, m_3$ are rational numbers. On the other hand $F(x,y)=0$, so every point $(x,y) \in S$ gives a rational point $(x,y,m_1,m_2,m_3)$ on the curve~$C_{123}$ in $\R^5$, given by
\begin{equation*}
\begin{aligned} 
F(x,y)=0,\\
z_1^2= \Big(\frac{2x-a_1-a_2}{2}\Big)^2+\Big(\frac{2y-b_1-b_2}{2}\Big)^2,\\
z_2^2=\Big(\frac{x+a_2-2a_1}{2}\Big)^2+\Big(\frac{y+b_2-2b_1}{2}\Big)^2 ,\\
z^2_3=\Big(\frac{x+a_1-2a_2}{2}\Big)^2+\Big(\frac{y+b_2-2b_1}{2}\Big)^2.
\end{aligned} 
\end{equation*}  
We use a similar argument to that in Theorem \ref{thm:perfect-triangle} to show that the genus of~$C_{123}$ is strictly bigger than one. In order to compute the genus of~$C_{123}$, we begin by considering the curves
\begin{equation*}
\begin{aligned}
C_1=\left\{(x,y,z_1): F(x,y)=0,\,z_1^2-\Big(\frac{2x-a_1-a_2}{2}\Big)^2-\Big(\frac{2y-b_1-b_2}{2}\Big)^2=0  \right\},\\
C_2=\left\{(x,y,z_2): F(x,y)=0, \, z_2^2-\Big(\frac{x+a_2-2a_1}{2}\Big)^2-\Big(\frac{y+b_2-2b_1}{2}\Big)^2=0 \right\},\\
C_3=\left\{(x,y,z_3):F(x,y)=0, \, z^2_3-\Big(\frac{x+a_1-2a_2}{2}\Big)^2-\Big(\frac{y+b_2-2b_1}{2}\Big)^2=0 \right\}.
\end{aligned}
\end{equation*}
Define the curve $C_{12}$ in $\R^4$ by 
\begin{equation*}\label{eq:third}
\begin{aligned}
F(x,y)=0,\\
z_1^2-\Big(\frac{2x-a_1-a_2}{2}\Big)^2-\Big(\frac{2y-b_1-b_2}{2}\Big)^2=0,\\
z_2^2-\Big(\frac{x+a_2-2a_1}{2}\Big)^2-\Big(\frac{y+b_2-2b_1}{2}\Big)^2=0.
\end{aligned}
\end{equation*}
Similarly, we may define $C_{13}$ and $C_{23}$. In the first step, we show that the genus of~$C_1$ is at least one. We can also show that the genus of $C_2$ and $C_3$ are at least one, but the proofs of these two cases is essentially the same as the proof for $C_1$, so we omit them.

Let $\pi_{1} \colon C_1 \rightarrow C$, $(x,y,z_1) \mapsto (x,y)$ be the projection onto the first two coordinates. The preimage of a point~$(x,y) \in C$ contains two distinct points
\[
\left(x,y, \pm z_1\right), \, \text{where} \quad z_1=\sqrt{\left(x-\frac{a_1+a_2}{2}\right)^2+\left(y-\frac{b_1+b_2}{2}\right)^2},
\]
except when $z_1=0$, which is the union of two lines $(x-\frac{a_1+a_2}{2})\pm i(y-\frac{b_1+b_2}{2})=0$ in $\C^2$. 

\textbf{The genus of $C_1$ is at least one:}
If $g(C)=1$, then it follows from Riemann-Hurwitz that 
\[
2g(C_1)-2 \ge 2(2-2) + \sum_{p \in C_i}(e_p-1), \quad \text{hence}, \quad g(C_1) \ge 1.
\]
If the genus of $C$ is zero, then it follows from Riemann-Hurwitz that 
\begin{equation*}\label{eq:rimann-hurwitz}
g(C_1) \ge -1+\frac{1}{2} \sum_{p \in C_i} (e_p-1).
\end{equation*}
So to get $g(C_1) \ge 1$ we need to show that the projection $\pi_1$ has at least~$3$ ramification points.

The potential ramification points correspond to the preimages of the intersection of $C$ with the lines $(x-\frac{a_1+a_2}{2})\pm i(y-\frac{b_1+b_2}{2})=0$, where by Bezout's Theorem there are $2d$ such points, counting with multiplicities in~$\p^2_{\C}$. Let $p$ be such an intersection point, then $p$ cannot be a ramification point if the curve has a singularity at $p$, or the curve $C$ is tangent to the line there. By varying $(a_1,b_1), (a_2,b_2)$ in $S$, we obtain infinitely many lines in the plane with slopes $\pm i$, each through the corresponding point~$(\frac{a_1+a_2}{2},\frac{b_1+b_2}{2})$, where only finitely many such lines are tangent to the curve $C$ or passing through its singularities. This is because the number of tangents that can be drawn from a fixed point in~$\p^2_{\C}$ to a given curve is finite.

On the other hand, we assumed that $S$ is an infinite set, thus for all but finitely many pairs of points $p_1, p_2 \in S$, the complex line $(x-\frac{a_1+a_2}{2})+i(y-\frac{b_1+b_2}{2})=0$ meets $C$ transversely at~$d$ points. Similarly, the complex line~$\left(x-\frac{a_1+a_2}{2}\right)-i\left(y-\frac{b_1+b_2}{2}\right)=0$ meets $C$ transversely at $d$ points.

If the middle point~$(\frac{a_1+a_2}{2}, \frac{b_1+b_2}{2})$ belongs to $C$, we get~$2d-2$ ramification points, and if the degree of~$C$ is at least~$3$ we obtain at least~$4$ ramification points, thus by Riemann-Hurwitz the genus of~$C_1$ is at least~$1$. 

If the degree of $C$ is $2$, then since $(a_1,b_1)$ and $(a_2,b_2)$ lie on $C$, we know that $(\frac{a_1+a_2}{2},\frac{b_1+b_2}{2})$ does not lie on~$C$, and in this case we get~$4$ ramification points. It follows from Riemann-Hurwitz that~$C_1$ has genus at least~$1$.  Hence the genus of $C_1$ always is at least one.

Now if one of the curves $C_{i}$ for~$i=1,2,3$ (say $C_1$) has genus at least~$2$, 
then we can determine a bound from below for the genus of $C_{123}$ using the Riemann–Hurwitz formula applied to the following projections
\[
{\displaystyle C_{123}\;{\xrightarrow {\ \rho_{1}\ }}\;C_{12}\;{\xrightarrow {\ \rho_{2}\ }}\;C_{1}\;{\xrightarrow {\ \rho_{3}\ }}} \displaystyle \, C,
\]
where each $\rho_i$ is a map of degree $2$. Therefore the genus of $C_{123}$ is at least two and by Faltings' Theorem $S$ must be a finite set, which is a contradiction. 

Now suppose that the genus of each $C_i$ is one, in this situation consider two $2:1$ projection maps $\pi_1 \colon C_1 \rightarrow C$ and $\pi_2 \colon C_2 \rightarrow C$ defined by $\pi_1((x,y,z_1))=(x,y)$ and $\pi_2((x,y,z_2))=(x,y)$ respectively. An equivalent definition of the curve $C_{12}$ is as follows,
\[
C_{12}\colon=\left\{(p_1,p_2) \in C_1 \times C_2: \pi_1(p_1)=\pi_2(p_2) \right\}.
\]
By the definition $C_{12}$ is the fiber product of $C_1$ and $C_2$ over $C$, and we have the following commutative diagram 
\[ \begin{tikzcd}
C_{12} \arrow{r}{\phi_1} \arrow[swap]{d}{\phi_2} & C_1 \arrow{d}{\pi_1} \\%
C_2 \arrow{r}{\pi_2}& C
\end{tikzcd}
\] 
Where, $\phi_1 (x,y,z_1,z_2)=(x,y,z_1)$ and $\phi_2(x,y,z_1,z_2)=(x,y,z_2)$. As we have seen, the branching points of $\pi_{1}$ correspond to the points on the intersection of two complex lines~$(x-\frac{a_1+a_2}{2})\pm i(y-\frac{b_1+b_2}{2})=0$ with $F(x,y)=0$, while the branching  points of $\pi_{2}$ correspond to the intersections of two complex lines 
\[
(x-2a_1+a_2)\pm i(y-2b_1+b_2)=0
\]
with $F(x,y)=0$.  Since these are lines with slopes $\pm i$ through distinct points $(\frac{a_1+a_2}{2},\frac{b_1+b_2}{2})$ and $(2a_1-a_2,2b_1-b_2)$, $\pi_{1}$ has at least one ramification point that is not a ramification point of $\pi_{2}$. This implies that each $\phi_i$ for $i=1,2$ has at least one ramification point that is not a ramification point. Indeed, let $y$ be a ramification point of $\pi_{1}$ that is not a ramification point of $\pi_{2}$, we have~$\pi_{1}^{-1}(y)=\left\{x_1\right\}$, while~$\pi_{2}^{-1}(y)=\left\{x_2,x_3\right\}$. Now if we assume that $\phi_i$ for $i=1,2$ has no ramification point, then we obtain~$(\pi_1 \circ \phi_1)^{-1} (y)=\left\{\alpha_1, \alpha_2\right\}$, where $\alpha_1 \not= \alpha_2$, while $(\pi_2\circ \phi_2)^{-1}(y)=\left\{\beta_1,\beta_2,\beta_3,\beta_4\right\}$, where the~$\beta_i$'s are distinct. On the other hand, by the commutativity of the above diagram we have $\pi_1\circ\phi_1=\pi_2\circ\phi_2$. Therefore $(\pi_1\circ \phi_1)^{-1}(y)=(\pi_2\circ \phi_2)^{-1}(y)$ and this is a contradiction. Thus Riemann-Hurwitz implies that
\[
2g(C_{12})-2\ge \deg(\phi_1)(2g(C_1)-2)+\sum_{p \in C_{12}}(e_p-1), \quad \text{hence}\quad g(C_{12})\ge 2.
\]
Therefore, by Faltings' Theorem $C_{12}$ has finitely many rational points. Hence the number of rational points on the curve~$C_{123}$ must be finite, otherwise by the projection from $C_{123}$ to~$C_{12}$ we obtain infinitely many rational points on $C_{12}$, and this is a contradiction.\hfill$\square$

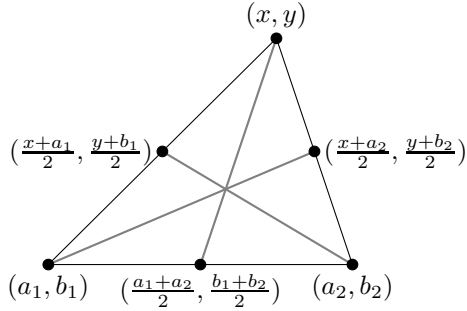
\begin{figure}
	\begin{tikzpicture}
	\draw (0,0) node[anchor=north]{$(a_1,b_1)$}
	-- (4,0) node[anchor=north]{$(a_2,b_2)$}
	-- (3,3) node[anchor=south]{$(x,y)$}
	-- cycle;
	\draw[gray, thick] (3,3) -- (2,0);
	\draw[gray, thick] (0,0) -- (3.5,1.5);
	\draw[gray, thick] (4,0) -- (1.5,1.5);
	\filldraw [black] (1.5,1.5) node[anchor=east]{$(\frac{x+a_1}{2},\frac{y+b_1}{2})$} circle (2pt);
	\filldraw [black] (3.5,1.5) node[anchor=west]{$(\frac{x+a_2}{2},\frac{y+b_2}{2})$} circle (2pt);
	\filldraw [black] (2,0) node[anchor=north]{$(\frac{a_1+a_2}{2},\frac{b_1+b_2}{2})$} circle (2pt);
	\filldraw [black] (3,3) node[anchor=west]{} circle (2pt);
	\filldraw [black] (0,0) node[anchor=west]{} circle (2pt);
	\filldraw [black] (4,0) node[anchor=west]{} circle (2pt);
	\end{tikzpicture}
	\caption{A triangle with its medians}\label{fig:triangle}
\end{figure}

\section{Final comments}
Similar to Shaffaf \cite{Shaffaf2018} and Tao \cite{Tao2014}, we can show that by assuming the weak Lang conjecture, if $S$ is a rational median set in the plane $\R^2$, then $S$ is a finite set. Moreover, there exists a natural number $N$ such that if $S$ is a rational median set then $|S| \le N$. 

\begin{question}
Does there exist a set of four non-collinear points in the plane $\R^2$, such that all its medians are rational? Furthermore, given a natural number $n$ can you find a rational median set of size $n$? 
\end{question}

\section{Acknowledgements}
The author was supported by the Austrian Science Fund (FWF): Project P 30405-N32.
I am grateful to Matteo Gallet, Niels Lubbes, Oliver Roche-Newton, Josef Schicho, and Audie Warren for several helpful conversations and comments.

\providecommand{\bysame}{\leavevmode\hbox to3em{\hrulefill}\thinspace}
\providecommand{\MR}{\relax\ifhmode\unskip\space\fi MR }
\providecommand{\MRhref}[2]{%
	\href{http://www.ams.org/mathscinet-getitem?mr=#1}{#2}
}
\providecommand{\href}[2]{#2}

\end{document}